\documentclass{amsart}
\usepackage[totalwidth=14cm,totalheight=22cm]{geometry}
\usepackage{amsfonts,amssymb,amsmath,amsthm}
\usepackage{url}
\usepackage{enumerate}
\usepackage{pst-node}
\usepackage{enumitem}

\usepackage{xcolor}
\usepackage[T1]{fontenc}
\usepackage{mathrsfs}
\usepackage{psfrag}
\usepackage{graphicx}
\usepackage{color}
 \usepackage{stmaryrd}
 
\newtheorem{theoremO}{Theorem}

\newtheorem{definition}{Definition}[section]

\newtheorem{theorem}[definition]{Theorem}

\newtheorem{lemma}[definition]{Lemma}
\newtheorem{proposition}[definition]{Proposition}
\newtheorem{remark}[definition]{Remark}
\newtheorem{corollary}[definition]{Corollary}
\newtheorem{fact}[definition]{Fact}
\newtheorem{question}[definition]{Question}

\numberwithin{equation}{section}

\renewcommand{\S}{\mathbb{S}}

\newcommand{\sa}{V_2}
\newcommand{\s}{\mathbf{stein}}

\newcommand{\M}{\mathcal{M}_{\geq 0}^1(\S^2)}

\newcommand{\K}{\mathcal{K}}

\newcommand{\Kl}{\mathbb{K}\mathbf{lein}^\infty}

\newcommand{\oshape}{\mathscr{H}\hspace{-0.06 cm}\emph{om}}

\newcommand{\shape}{\mathscr{S}\hspace{-0.06 cm}\emph{hape}}

\newcommand{\sob}{H^1(\mathbb{S}^{n-1})}

\newcommand{\R}{\mathbb{R}}

\newcommand{\N}{\mathbb{N}}

\newcommand{\w}{\mathbf{mean}}

\newcommand{\A}{\overline{V_2}^n}

\newcommand{\Vb}{\overline{V_1}^n}

\renewcommand{\d}{d_{\mathscr{H}^n}}

\newcommand{\dpk}{d_{\mathscr{H}^{n+p}}}

\newcommand{\dd}{d_{\mathscr{H}^2}}

\newcommand{\di}{d_{\mathscr{H}^\infty}}

\newcommand{\bd}{d_{\mathscr{S}^n}}
\newcommand{\bdpk}{d_{\mathscr{S}^{n+p}}}

\newcommand{\bdd}{d_{\mathscr{S}^2}}

\newcommand{\bdi}{d_{\mathscr{S}^\infty}}

\newcommand{\supp}{\operatorname{Supp}}

\title{Hyperbolic geometry of shapes of  convex bodies}

\author{Cl\'ement Debin and Fran\c{c}ois Fillastre}
\address{Universit\'e Grenoble--Alpe}
\address{CY Cergy Paris Universit\'e / Universit\'e de Montpellier}

\date{\today v2}



\begin{document}

\maketitle

\begin{abstract}
We use the intrinsic area  to define a distance on the space of homothety classes of convex bodies in the $n$-dimensional Euclidean space, which makes it isometric to a convex subset of the infinite dimensional hyperbolic space. The ambient Lorentzian structure is an extension of the intrinsic area form of convex bodies, and  Alexandrov--Fenchel Inequality is interpreted as the Lorentzian reversed Cauchy--Schwarz Inequality. 

We deduce that the space of similarity classes  of convex bodies  has a proper geodesic distance with curvature bounded from below by $-1$ (in the sense of Alexandrov). In dimension $3$, this space is homeomorphic to the space  
of distances with non-negative curvature on the $2$-sphere, and this latter space contains the space of flat metrics on the $2$-sphere considered by W.P.~Thurston. Both Thurston's and the area  distances rely on the area form. So the latter may be considered as a  generalization of the "real part" of Thurston's construction.

\end{abstract}


\section{Introduction}

 Let $P$ be a non-empty space of flat metrics on the $2$-sphere,  with $n>3$ prescribed angles $0<\alpha_i<2\pi$ at the cone singularities, up to orientation-preserving similarities, and with a labeling of the cone-points. In a celebrated article \cite{th}, W.P.~Thurston uses the area of the flat metrics to endow $P$ with a \emph{complex} hyperbolic structure. Among the multitude  generalizations and adaptations of this construction, let us consider  subspaces of $P$ endowed with an isometric involution, studied in \cite{BG}. They are isometric to     spaces of homothety classes of plane convex polygons with fixed direction of edges, endowed with  real hyperbolic distances. This latter point of view was then extended to any dimension, using mixed-volumes to hyperbolize some spaces of convex polytopes in $\R^n$. For $n=3$, some of these spaces, which are isometric to (real!) hyperbolic polyhedra, isometrically embeds into $P$ \cite{FI,FIgr}. 

\begin{figure}[h!]
\begin{center}
\includegraphics[width=0.9\linewidth]{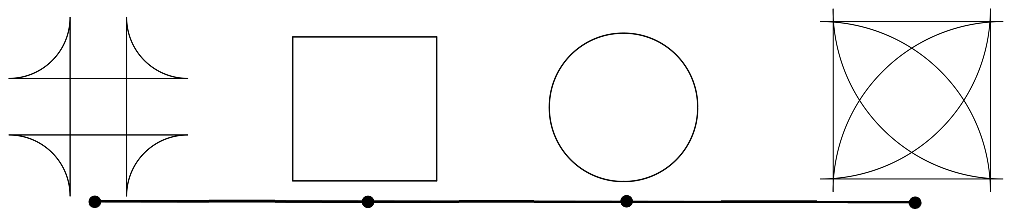}\caption{
The convex segment between the disc and the square is extended until we arrive at two objects of zero (formal) area. The hyperbolic distance between the disc and the square is half of the logarithm  of the cross-ratio of the four points.}\label{fig:premierdessin}
\end{center}
\end{figure} 

In the first part of the present article, we bring this real hyperbolization process to its full generality, by endowing the space of convex bodies in $\R^n$ with an ``area distance'', which appears to be hyperbolic in a sense clarified below.
The idea behind the definition of the area distance is quite natural. Consider  the convex combination $K_t=tK_1+(1-t)K_2$, of two convex bodies,  $t\in[0,1]$. In general, by Alexandrov--Fenchel Inequality, there exists $t_0,t_1\in \R$, $t_0\leq 0<1\leq t_1$ such that 
the formal area of $K_t$ is zero.  
 We then have two points ($0$ and $1$) on the segment $[t_0,t_1]$, and, heuristically, $t_0$ and $t_1$ belong to the isotropic cone of a quadratic form (the area). Mimicking the definition of the distance of the Klein model of the hyperbolic space, we define the area distance as half of the log  of the cross-ratio of $t_0,0,1,t_1$, see Figure~\ref{fig:premierdessin}. See also Figure~\ref{fig:quadrangl}.  The precise definition of the area distance will be given in   Section~\ref{sec intr area}.

  Recall that two subsets $A$ and $B$ of $\R^n$ are \emph{homothetic} if they differ by a translation and a positive scaling.
If $K$ is a convex body, we denote by $[K]$ its homothety class, and by 
 $\oshape^{n*}$ the space of homothety classes of all the convex bodies in $\R^n$, which are different from points and segments.
The area distance introduced above is clearly invariant under homotheties.  
  Let us denote by  $\d$ the induced area  distance on  $\oshape^{n*}$.
Note that it is not obvious that this is actually a distance.

\begin{theoremO}\label{main}
$(\oshape^{n*},\d)$ is a metric space which 
\begin{enumerate}[nolistsep]
\item is uniquely geodesic, and the unique shortest path between $[K_1]$ and $[K_2]$ is  the class of the convex combination of $K_1$ and $K_2$,
\item is of infinite Hausdorff dimension and infinite diameter,
\item is proper,
\item has  curvature bounded from below and above by $-1$ in the sense of Alexandrov,
\item has boundary homeomorphic to the real projective space of dimension $(n-1)$,
\item \label{prop ext} any point is the endpoint of a shortest path that is not extendable beyond this point,
\item \label{prop top}  is homeomorphic to  the space of convex bodies of intrinsic area equal to one and Steiner point at the origin, endowed with the Hausdorff distance.
\end{enumerate}
\end{theoremO}

As some definitions may depend on the authors, let us recall that a metric space is \emph{geodesic} if any two points are joined by a shortest path, it is \emph{uniquely geodesic} if the shortest path is unique;
and it is \emph{proper} if every bounded closed subset is compact. A proper  metric space is locally compact and 
complete. A shortest path is \emph{extendable} if it is strictly contained in another shortest path. The boundary of a metric space  is the set of equivalence classes of geodesic rays at bounded distance, endowed with a natural topology, see \cite{BH} for details. In the present article, the definition of bounded curvature in the sense of Alexandrov is global.

The property \eqref{prop ext} is proved in Section~\ref{sec terminal}. The topological properties in Theorem~\ref{main} are  consequences of   a theorem of R.A.~Vitale and the Blaschke Selection Theorem, see Section~\ref{haus}. 
The other assertions in Theorem~\ref{main} are either straightforward, 
or they come from the following  extrinsic description of $(\oshape^{n*},\d)$.

\begin{theoremO}\label{main1b}
 $(\oshape^{n*},\d)$ is isometric to an infinite dimensional unbounded closed convex subset with empty interior of  the infinite dimensional hyperbolic space. 
\end{theoremO}

Here, ``the'' infinite dimensional hyperbolic space is defined from a separable Hilbert space.
 The isometry in Theorem~\ref{main1b} is obtained by considering the support function of convex bodies. Under this identification, the area of convex bodies  will give a bilinear form, that appears to have a Lorentzian signature. This is actually very natural, as for example, Alexandrov--Fenchel Inequality for mixed-area is then given by a reversed Cauchy--Schwarz Inequality. 
 
We say that the distance $\d$ is hyperbolic, because it is 
isometric to a totally geodesic subspace of an hyperbolic space, 
or because of the curvature property (4) in  Theorem~\ref{main} (the latter being an immediate consequence of the former). Note that for metric spaces, it is meaningless to speak about ``curvature equal to $-1$''.

It was pointed out by Nicolas Monod  to the second author that  the present construction for $n=2$  gives an explicit example of an exotic action of $PSL(2,\R)$ on the infinite dimensional hyperbolic space \cite{MP}.

\bigbreak

In the second part of the present article, we investigate  $\shape^{n*}$,  the quotient of $\oshape^{n*}$ by linear isometries of the Euclidean space $\R^n$: $\shape^{n*}$ is the space of convex bodies in $\R^n$ (not reduced to points or segments) up to Euclidean similarities (such an equivalence class  is the ``shape'' of the convex body). It is endowed with the quotient distance $\bd$. We obtain the following.
 \begin{theoremO}\label{main2}
$(\shape^{n*},\bd)$ is a proper geodesic metric space with  curvature $\geq -1$ and with boundary reduced to a single point. It is not uniquely geodesic. It contains many totally geodesic hyperbolic surfaces.
\end{theoremO}

There is another complex hyperbolic orbifold considered by Thurston,  which is defined similarly to the space $P$ introduced at the beginning of the present article, but where the singular points are not labeled. It is a subspace of $\M$, the space of metrics of non-negative curvature on the sphere, up to isometries, and with unit area. A natural  generalization of Thurston construction would be  to use  the area of the  metrics to endow $\M$ with a distance, and look at its properties. For example, one may look at curvature properties, or possible complex structure. From  \cite{Bel} and  \eqref{prop top} in Theorem~\ref{main}, it follows that  $\shape^{3*}$ and $\M$ are homeomorphic, if the latter space is endowed with the topology of uniform convergence of distances.\footnote{For $n\geq 3$, the induced inner distance on the boundary of a convex body in $\R^n$ is (isometric to) a 
distance of non-negative curvature on $\mathbb{S}^{n-1}$ in the sense of Alexandrov. But not every such distance of non-negative curvature on $\mathbb{S}^{n-1}$ can 
arise in this way 
(\cite{matv}, \cite[1.9]{alex}).} So  Theorem~\ref{main2}, for $n=3$, may be seen as a  ``real hyperbolization'' of $\M$ with its natural topology.  Here the word ``hyperbolization'' is used in a wide sense, as  as $(\shape^{n*},\bd)$ is not uniquely geodesic, it is not of non-positive curvature, hence not with curvature $\leq -1$. However that's an open question to  know if it is locally of non-positive curvature.

We conclude the present article by a question about the space of shapes of \emph{all}  convex bodies (regardless of the dimension of the ambient space).

\bigbreak

As we pointed out, the idea to consider convex bodies in an ambient hyperbolic space came from the observation that  the Alexandrov--Fenchel Inequality for the mixed-area of convex bodies
looks like the reversed Cauchy--Schwarz Inequality in a Lorentzian vector space (see Remark~\ref{rem ref}). In dimension $2$, Alexandrov--Fenchel Inequality coincides with the Minkowski inequality. Also, mixed-volumes were introduced by Minkowski. He also introduced Lorentzian vector spaces, which are now called Minkowski spaces. We are not aware if Minkowski knew a relation between the inequality and the spaces that both bear his name. But as far as we know, it seems that in the meantime this relation between the fundamental inequality of the theory of convex bodies and basic Lorentzian geometry was forgotten.

\subsection*{Acknowledgements.}
The authors want to thank Nicola Gigli, Julien Maubon, Nicolas Monod,  Graham Smith, Pierre-Damien Thizy and Giona Veronelli for useful conversations. They also want to thank Igor Belegradek who helped to clarify some points in a preceding version of the text. This work was completed during a visit of the second author at SISSA. He wants to thank the institution for its hospitality.

\section{The area distance}

\subsection{Intrinsic area of convex bodies}\label{sec intr area}

A \emph{convex body} is  a non-empty compact convex subset of $\mathbb{R}^n$. In the present article, we set $n>1$.
For a plane convex body $K$ (i.e. a convex body in $\R^2$), speaking about the ``area''  of $K$ usually means to look at its volume (two  dimensional Lebesgue measure). Note that the area of plane convex bodies is positively homogeneous of degree $2$: for $\lambda>0$, 
$\operatorname{vol}_2(\lambda K)=\lambda^2 \operatorname{vol}_2(K)$. For a convex body in $\R^3$,  the ``area'' usually refers to its surface area, i.e. the $2$-dimensional total Hausdorff measure of its boundary $\partial K$.
 Here also, the surface area is 
positively homogeneous of degree two. 

For $n>3$, there are two ways to generalize the notion of ``area'' to convex bodies in $\R^n$. 
Both are coming from the 
 Steiner Formula.
 Let $B^n$ be the closed unit ball centered at the origin in $\R^n$, and let  $\kappa_n$ be its volume. Let us set $\kappa_0=1$ and $\kappa_1=2$. If $K$ is a convex body in $\R^n$, then there exist non-negative real numbers $V_i(K)$, $i=0,\ldots,n$ such that, for any $\epsilon >0$,
  \begin{equation}\label{steiner}\mathrm{vol}_n(K+\epsilon B^n)=\sum_{i=0}^n \epsilon^{n-i}\kappa_{n-i}V_i(K)~.\end{equation}
Here $\mathrm{vol}_n$ is the Lebesgue measure of $\R^n$, and the sum is the Minkowski addition: $A+B=\{a+b|a\in A,b\in B\}$.
It appears that $V_0(K)=1$ and $V_n(K)=\mathrm{vol}_n(K)$. 

The first way to generalize the notion of surface area of convex bodies in $\R^3$ is to consider $V_{n-1}(K)$ as the ``area'', given by the first order variation of $\operatorname{vol}_n(K+\epsilon B^n)$, seen as a function of $\epsilon$. Note that this ``area'' is homogeneous of degree $(n-1)$, and that for $n=2$, this is related to the perimeter of the convex body and not to its area. 

In the present article, we consider another way to generalize the notion of surface area of convex bodies in $\R^3$, and  we call $V_2(K)$ given by \eqref{steiner} the \emph{intrinsic area} of $K$. Let us mention some relevant properties. The property \ref{mcmullen} explains the terminology ``intrinsic''.

\begin{enumerate}[label={A\arabic*)},nolistsep]
\item \label{homo} For any $\lambda >0$, $\sa(\lambda K)=\lambda^2\sa(K)$;
\item \label{prop 1V2} $V_2(K)\geq 0$;
\item \label{propV22} $K_1\subset K_2 \Rightarrow V_2(K_1) \leq V_2(K_2)$;
\item  \label{propV23} $V_2(K)=0$ if and only if $K$ is a point
or a segment;
\item \label{isom} for any $A\in O(n)$ and $p\in \R^n$,
$V_2(A(K)+\{p\})=V_2(K)$;
\item \label{mcmullen}
Let $\iota:\R^n\to \R^{n+1}$ be a linear isometric embedding. 
Then $V_2(\iota(K))=V_2(K)$.
\end{enumerate}

%

The (intrinsic) area can be ``polarized'', in the sense that there exists a function called the (intrinsic) \emph{mixed-area} $\sa(\cdot,\cdot)$, that can be defined as
\begin{equation}\label{def mixed v2}
V_2(K_1,K_2)=\frac{1}{2}\left(\sa(K_1+K_2)-\sa(K_1)-\sa(K_2) \right)~,
\end{equation} 
and satisfies the following properties:
\begin{enumerate}[label={M\arabic*)},nolistsep]
\item \label{mixed 1} $\sa(K_1,K_1)=\sa(K_1)$;
\item \label{mixed 2} $\sa(K_1,K_2)=\sa(K_2,K_1)$;
\item \label{mixed 3} $\sa(K_1+K_2,K_3)=\sa(K_1,K_3)+\sa(K_2,K_3)$;
\item \label{mixed 4} for $\lambda >0$, $\sa(\lambda K_1,K_2)=\lambda  \sa(K_1,K_2)$;
 \item \label{prop mixed incl} $K_1\subset K_2 \Rightarrow V_2(K_1,K_3)\leq V_2(K_2,K_3)$;
\item \label{isotrope} $K$ is a point if and only if  for any convex body $Q$, 
$\sa(K,Q)=0$;
\item \label{mixed 2'} $V_2(K_1,K_2) \geq 0$; and $V_2(K_1,K_2)=0$ if and only if $K_1$ or $K_2$ is a point, or both are segments with the same direction;
\item \label{alex fench} we have
\begin{equation} \label{eq alex fench}
\delta(K_1,K_2)=\sa(K_1,K_2)^2 -  \sa(K_1)\sa(K_2)\geq 0
\end{equation}
and if $K_1$ and $K_2$ are not points, then equality occurs if and only if $K_1$ and $K_2$ are homothetic.
\end{enumerate}

All those properties are classical, as $V_2$ is a particular case of mixed-volume: $V_2(K_1,K_2)=V(K_1,K_2,B^n,\ldots,B^n)$ 
 \cite{schneider}. Property M8) is Alexandrov--Fenchel Inequality. In the present article, we will generalize the properties listed above, using some simple analysis of functions on the sphere. Before that, let us introduce the area distance on the space of homothety classes of convex bodies.
We will give two equivalent definitions, both using Alexandrov--Fenchel Inequality  \ref{alex fench}.

In the sequel, we denote by $\mathcal{K}^n$ the set of convex bodies in $\R^n$, and by $\mathcal{K}^{n*}$ the subset of convex bodies of positive intrinsic area. In other terms,  by \ref{prop 1V2} and  \ref{propV23},  $\mathcal{K}^{n*}$  is $\mathcal{K}^n$ minus points and segments.
By property \ref{alex fench} of the mixed-area, for any $K_1,K_2 \in \mathcal{K}^{n*}$, the quantity
$$\tilde{d}_1(K_1,K_2) = \operatorname{argch}\left( \frac{V_2(K_1,K_2)}{\sqrt{V_2(K_1)V_2(K_2)}}\right)~
$$
is well-defined. This is also clear that $\tilde{d}_1(K_1,K_2)$ is invariant under  positive scaling of $K_1$ and $K_2$. 
Moreover, by \ref{isom} and \eqref{def mixed v2},
for all $p\in\R^n$,
$$V_2(K_1+\{p\},K_2)=V_2(K_1,K_2+\{p\})=V_2(K_1,K_2)~, $$
hence $\tilde{d}_1$ is invariant under translations of $K_1$ or $K_2$. By the case of equality in property \ref{alex fench}, $\tilde{d}_1(K_1,K_2)=0$ if and only if $K_1$ differ from $K_2$ by a homothety.

Let us define the space $\oshape^{n}$ (resp. $\oshape^{n*}$) as the quotient of $\mathcal{K}^{n}$ (resp. $\mathcal{K}^{n*}$) by homotheties. For a convex body $K$, we denote by $[K]$  the set of homothetic copies of $K$.
For any $[K_1],[K_2] \in \oshape^{n*}$ we set
$$ d_1([K_1],[K_2])=\tilde d_1(K_1,K_2)~.$$

 Let us do it in a different way. Let $K_1,K_2 \in \mathcal{K}^{n*}$. Assume that $V_2(K_1)=V_2(K_2)=a>0$ and that $[K_1] \neq [K_2]$. Consider the following equation:
\begin{equation}\label{equation}
\sa((1-t)K_1+t K_2)=0~.
\end{equation}
By properties of the mixed-area, the left-hand side is a polynomial in $t$, and the coefficient of $t^2$ is $2a-2V_2(K_1,K_2)$.
Since $[K_1] \neq [K_2]$, by Alexandrov-Fenchel Inequality \ref{alex fench}, we have $V_2(K_1,K_2) > a$: the coefficient of $t^2$ is negative, in particular this is a second order polynomial. An easy calculation shows that its discriminant is equal to $4\delta(K_1,K_2) >0$ (see \eqref{eq alex fench}). Let $t_1<0<1<t_2$ be the two real solutions of the equation \eqref{equation}, and let us define 
$$\tilde{d}_2(K_1,K_2)=\frac{1}{2} \ln [0,1,t_1,t_2]~,$$
 where  $[0,1,t_1,t_2]=\frac{t_1}{t_2}\frac{1-t_2}{1-t_1}$  is the cross-ratio.

By \eqref{equation}, it is clear that $\tilde{d}_2$ is invariant by translation of $K_1$ or $K_2$. Let $[K_1], [K_2]\in \oshape^{n*}$, and let $K_1, K_2$ be two 
representatives having the same intrinsic area. We can then define
$$d_2([K_1],[K_2])=\tilde{d}_2(K_1,K_2)~,$$
if $[K_1]\not= [K_2]$, and zero otherwise.

Classical trigonometry computations from hyperbolic geometry show $d_1=d_2$. We define the \emph{area distance} on $\oshape^{n*}$ as $$\d:=d_1=d_2~.$$ (Note that we didn't proved yet that it is a distance.)

Even if the space of convex bodies is not a  vector space, from its properties the mixed-area reminds a symmetric bilinear form, whose kernel is the space of  points, and whose isotropic cone is the space of points and segments. Moreover, Alexandrov--Fenchel Inequality \eqref{eq alex fench} reminds a reversed Cauchy--Schwarz Inequality. To define $d_1$ and $d_2$ above, we mimicked the definitions of the hyperboloid model and the Klein model of the hyperbolic space. It is actually the way we will prove Theorem~\ref{main}.

\subsection{Spaces of support functions}\label{sec supp}

\begin{figure}
\begin{center}
\psfrag{C}{$\mathcal{C}_n$}
\psfrag{K}{$\Kl_n$}
\psfrag{H}{$\mathcal{H}_n^\infty$}
\psfrag{H+}{$\w>0$}
\psfrag{H01}{$\sob_{01}$}
\psfrag{L}{$\operatorname{L}$}
\psfrag{1}{}
\psfrag{c}{}
\psfrag{P}{}
\psfrag{V1}{$\Vb=1$}
\psfrag{V}{$\Vb=0$}
\includegraphics[width=0.5\linewidth]{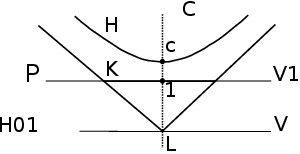}\caption{Notations for subspaces of $\sob_1$.}\label{fig:notations}
\end{center}
\end{figure}

The \emph{support function} $\supp(K)$ of a convex body $K$ in $\R^n$ gives, at the point $x\in \S^{n-1}$, the distance from the origin of $\R^n$ to the support hyperplane of 
$K$ with outward normal $x$. More precisely, 
$\supp(K):\S^{n-1}\to \R$ is defined as
$$\supp(K)(x)=\max_{p\in K} \langle x,p\rangle~, $$
where $\langle\cdot,\cdot\rangle$ is the usual  scalar product of $\R^n$. 

Let us denote by $\|\cdot\|_{L^2}$ the $L^2$ norm on the round sphere $\S^{n-1}$. 
Let $\sob$ be the Sobolev space of $\S^{n-1}$, i.e. the space  of functions $\S^{n-1}\to \R$ which are in $L^2(\S^{n-1})$
 as well as their first order derivatives in the weak sense. The space $\sob$ is implicitly endowed with  the norm
$$\|h\|_{H^1}=\left(\|h\|_{L^2}^2 + \| \nabla h\|_{L^2}^2\right)^{1/2}
=\left(\int_{\mathbb{S}^{n-1}}h^2+\|\nabla h\|^2\right)^{1/2} $$
where the gradient $\nabla$ is the one of the round sphere. 

If $K$ is contained in the ball centered at the origin and with radius $R$, then  $\supp(K)$ is $R$-Lipschitz.
Hence we get a map
$$\supp : \mathcal{K}^n\to \sob~. $$

Let us recall some basic properties  \cite{schneider,groemer,heil}:
\begin{itemize}[nolistsep]
\item a function $h:\mathbb{S}^{n-1}\to \R$ is the support function of a convex body in $\R^n$ if and only if its one homogeneous extension $\tilde{h}:\R^n\setminus\{0\}\to \R$, $\tilde{h}(x)=\|x\|h(x/\|x\|)$, $\tilde{h}(0)=0$,  is a convex function;
\item $\supp(K_1+K_2)=\supp(K_1)+\supp(K_2)$, $\supp(\lambda K)=\lambda \supp(K)$, $\lambda>0$; in particular, 
$\supp(\mathcal{K}^n)$ is a convex cone in $\sob$.
\item $\supp$ is a bijection onto its image; 
\item if $K_1\subset K_2$, then $\supp(K_1)\leq \supp(K_2)$;
\item if $(\supp(K_i))_i$ converges pointwise to $\supp(K)$, then
the convergence is uniform;
\item if $(\supp(K_i))_i$ converges  to $\supp(K)$, then almost everywhere
$(\nabla \supp(K_i))_i\to \nabla \supp (K)$.
\end{itemize}

\begin{remark}{\rm 
Let us warn the reader that if $\supp(\lambda K)=\lambda \supp(K)$, $\lambda>0$, we don't have $\supp(- K)=- \supp(K)$ in general, where $-K=\{-x | x\in K \}$. Indeed, both $\supp(- K)$ and $\supp(K)$ are positive if the origin of $\R^n$ is in the interior of $K$. 
Actually, $\supp(-K)(v)=\supp(K)(-v)$, and $- \supp(K)$ is like the support function of $K$, but with the  support planes defined by their  \emph{inward} unit normals.
}\end{remark}

Let us set $\lambda_1=n-1$ and $c_n$ be a given positive constant.  
For $h\in \sob$, let us consider the quadratic form 
\begin{equation}\label{eq def A}\A(h)=c_{n}\left(\|h\|_{L^2}^2 -\lambda_1^{-1} \|\nabla h\|_{L^2}^2\right)~, \end{equation}
that comes from  the following bilinear form: for $h,k\in \sob$,
$$\A(h,k)=c_{n} \left( \mathbf{(}h,k\mathbf{)}_{L^2} - \lambda_1^{-1}(\nabla h,\nabla k)_{L^2}\right)~.$$
To avoid confusion, let us emphasis that $\A(h,h)=\A(h).$ 
It is known (see e.g. \cite[Theorem~4.2a]{heil}, \cite[p.~298]{schneider} or \cite[Proposition~2.4.2]{groemer}) that for any $n$ there is a unique $c_n$ such that, for any $K_1, K_2\in \mathcal{K}^n$
$$V_2(K_1,K_2)=c_n\A(\supp(K_1),\supp(K_2))~. $$

Let us first restrict $\A$ to a subspace where it is not degenerate. Hopefully, the kernel of $\A$  is exactly the image of points by $\supp$. Indeed, the support function of the point $p\in \R^n$ is the restriction to the sphere of the linear map $x\mapsto \langle p,x\rangle$. But the space of such maps is the eigenspace of the first non-zero eigenvalue of the Laplacian on the round sphere, and this eigenvalue is the $\lambda_1$ in \eqref{eq def A}, so we deduce easily the following fact.

\begin{fact}\label{noyau A}
The kernel of $\A(\cdot,\cdot)$ on $\sob$ is the eigenspace of $\lambda_1$. 
\end{fact}

\begin{proof}
Let $h \in \sob$. The function $h$ belongs to the kernel of $\A(\cdot,\cdot)$ if and only if for any $k \in \sob$ we have
$$
\int_{\S^{n-1}} hk = \lambda_1^{-1} \int_{\S^{n-1}}  \langle \nabla h, \nabla k \rangle~.
$$
By density of smooth functions on $\S^{n-1}$ for the $H^1$-norm and by Green Formula, this is equivalent to the following property: for any smooth function $k$ on $\S^{n-1}$ we have
$$
\int_{\S^{n-1}} hk = \lambda_1^{-1} \int_{\S^{n-1}} h \Delta k~,
$$
and this means $h = \lambda_1^{-1} \Delta h$ in the weak (hence smooth) sense.
\end{proof}

 We will denote by $\sob_1$ the subspace 
of $\sob$ of functions $L^2$-orthogonal 
to the eigenspace of $\lambda_1$, i.e.
\begin{eqnarray*}
\sob_1 & = & \{h\in \sob | \left( h,x^i\right)_{L^2}=0~, i=1,\ldots,n \} \\
& = & \{h\in \sob | \int_{\S^{n-1}} h(x)x \,\mbox{d}\S^{n-1}(x)=0\}~. 
\end{eqnarray*}
In turn, 	$\A$ is non-degenerate on $\sob_1$. This space has a clear geometric meaning for convex bodies.
Recall that the Steiner point of a convex body $K$ is the following point of $\R^n$:
$$\s(K)=\frac{1}{\kappa_n}\int_{\mathbb{S}^{n-1}} \supp(K)(x)x\mbox{d}\mathbb{S}^{n-1}~, $$
so that
$$\s(K)=0 \Longleftrightarrow \supp(K)\in \sob_1~.$$
We have that for any $p\in \R^n$, $\s(K+\{p\})=\s(K)+\{p\}$, hence a convex body with Steiner point at the origin is a representative of the class of this convex body up to translations.

Now we prove that $\A$ has a Lorentzian signature on $\sob_1$: it is positive in one direction, and negative-definite on the orthogonal (for a given scalar product, here the Sobolev one).
Let $\operatorname{L}$ be the line of constant functions in $\sob_1$. We denote by $\sob_{01}$ the subspace of $\sob_1$ of elements $H^1$ (or, equivalently, $L^2$) orthogonal to  $\operatorname{L}$.

\begin{lemma} \label{lemme equiv distance} For $h\in \sob_{01}$,
\begin{equation}\label{lem:laplacien}
c_{n}\left(\frac{\lambda_2-\lambda_1}{\lambda_1}\right)\|h\|_{L^2}^2 \leq -\A(h)\end{equation}
 and
\begin{equation}\label{lem:laplacienH1}
c_{n} \left( \frac{\lambda_2 - \lambda_1}{\lambda_1\lambda_2} \right) \|h\|^2_{H^1} \leq - \A(h) \leq c_{n}\frac{1}{\lambda_1} \|h\|^2_{H^1}~.
\end{equation}
\end{lemma}

\begin{proof}
The space $\operatorname{L}$ is exactly the eigenspace of the zero eigenvalue of the spherical Laplacian. If we denote by $\lambda_2(>\lambda_1)$ the second positive eigenvalue, then 
by Rayleigh Theorem,
 for $h\in \sob_{01}\setminus \{0\}$ we have
\begin{equation}\label{rayleigh2}\lambda_2 \leq \frac{ \| \nabla h \|_{L^2}^2}{\|h\|_{L^2}^2}~.\end{equation}

Now \eqref{lem:laplacien} is immediate from
\eqref{rayleigh2}, and
the right-hand side inequality in \eqref{lem:laplacienH1} follows from 
$$-\A(h)\leq c_{n}\lambda_1^{-1}\|\nabla h\|^2_{L^2}\leq c_{n} \lambda_1^{-1}\|h\|^2_{H^1}~. $$
The left-hand side inequality in \eqref{lem:laplacienH1} follows by adding the two following inequalities:
as $\lambda_2>\lambda_1=n-1 \geq 1$, \eqref{lem:laplacien} gives
$$c_{n}\frac{1}{\lambda_2}\left(\frac{\lambda_2-\lambda_1}{\lambda_1}\right)\|h\|^2_{L^2}\leq -\A(h)~, $$
and
on the other hand, using again \eqref{rayleigh2}, the equality \eqref{eq def A} gives
$$c_{n}\left(\frac{1}{\lambda_1 }-\frac{1}{\lambda_2}\right)\|\nabla h\|_{L^2}^2\leq -\A(h)~. $$

\end{proof}
 
Clearly $\A$ is positive definite on $\operatorname{L}$, and we have:

\begin{proposition}\label{prop:hilbert}
$(\sob_{01},-\A(\cdot,\cdot))$ is a separable Hilbert space.
\end{proposition}
\begin{proof}
By \eqref{lem:laplacien} or \eqref{lem:laplacienH1}, 
$-\A$ is a scalar product on $\sob_{01}$. 
As $\sob_{01}$ is orthogonal  to a vector subspace, it is a closed subspace, hence complete and separable for the $H^1$ norm.  The result follows from \eqref{lem:laplacienH1}.
  \end{proof}

Note that as $\A$ is Lorentzian on $\sob_1$, we obtain the reversed Cauchy--Schwarz Inequality, that generalizes  Alexandrov--Fenchel Inequality \ref{alex fench}:
\begin{equation}\label{RCS}\A(h,k)^2 \geq \A(h)\A(k)~, \end{equation} 
for  $h,k\in \mathcal{C}_n$ with 
(see Figure~\ref{fig:notations})
\begin{equation*}\label{def cone}\mathcal{C}_n=\{h\in \sob_1 | \A(h)> 0, \Vb(h)>0 \}~,\end{equation*}
and where $$\Vb(h)=\frac{1}{\kappa_{n-1}}\int_{\S^{n-1}}h~,$$
and equality occurs in  \eqref{RCS} if and only if $h=\lambda k$, $\lambda >0$.

Let us mention that it is known that, for a convex body $K\subset \R^n$, if $V_1(K)$ is given by \eqref{steiner}, then
$$V_1(K)=
\Vb(\supp(K))~.$$

\subsection{Infinite dimensional hyperbolic space}

Let us introduce \begin{equation*}\label{hinf}
\mathcal{H}_n^\infty=\{h \in \mathcal{C}_n | \A(h)=1 \}~.
\end{equation*}

As the Hilbert structure on $\sob_{01}$ is given by $\A$, the map $\A$ is smooth, and it is easy to see that $\mathcal{H}_n^\infty$ is the graph of a smooth map over $\sob_{01}$, hence an infinite dimensional smooth manifold. 
We implicitly endow $\mathcal{H}_n^\infty$ with the restriction of $-\A( \cdot,\cdot)$ on its tangent spaces.
The intersection of $\mathcal{H}_n^\infty$ with any vector subspace of finite dimension $p$ of $\sob_1$ containing a vector of $\mathcal{C}_n$, is clearly a hyperboloid model of the hyperbolic space of dimension $(p-1)$. In turn,
$\mathcal{H}_n^\infty$  is a Riemannian manifold of constant sectional curvature $-1$. 
Moreover, it is not hard to see that the map $\mathbf{p}_{\mathcal{H}} : \mathcal{H}_n^\infty \rightarrow \sob_{01}$,
 $\mathbf{p}_{\mathcal{H}}(h)=h-\frac{\Vb(h)}{\Vb(1)}$ is bijection and locally bi-Lipschitz, so by Proposition~\ref{prop:hilbert}, $\mathcal{H}_n^\infty$ is complete.

Let us denote by  $d_{\mathcal{H}}$ the distance induced by the Riemannian structure, and we have, in the same way than  in the finite dimensional case,
\begin{equation*}\label{eq: dist hyperboloide}d_{\mathcal{H}}(h,k)=\operatorname{argch} \A\left(h,k \right)~. \end{equation*}  

We will also need the pull-back of the distance on the hyperboloid onto $$\Kl_n=\{h \in \mathcal{C}_n ~|~ \Vb(h)=1\}~$$ 
via a central projection, i.e. the hyperbolic distance on $\Kl_n$ is defined by
\begin{equation}\label{eq:dk dh}
d_{\mathbb{K}}(h,k) :=  d_{\mathcal{H}}(\A(h)^{-1/2}h,\A(k)^{-1/2}k)~.
\end{equation}

Of course it is possible to write $d_{\mathbb{K}}$ in an intrinsic way, as we did in  Section~\ref{sec intr area} for the area distance, using \eqref{RCS} instead of \ref{alex fench}.
For future references let us note the  following non-surprising facts, whose proofs are left to the reader.

\begin{fact} \label{cor topo hyperboloid}
 On $\mathcal{H}_n^\infty$, $d_\mathcal{H}$ and $d_{H^1}$ induce the same topology, where $d_{H^1}$ is the distance induced by $\|\cdot\|_{H^1}$.
\end{fact}

\begin{fact}\label{lem bord infini}
Let $h_i, k \in \Kl_n$. Then
$$\A(h_i)\to 0 \iff d_\mathbb{K}(h_i,k)\to +\infty~.$$
\end{fact}

\begin{fact}\label{lem conv klein}
Let $(h_i)_i$ converge to $h$ in $(\Kl_n,d_{\mathbb{K}})$. Then $\A(h-h_i)\to 0$.
\end{fact}

\begin{fact} \label{prop equiv topo Klein}
On $\Kl_n$, $d_\mathbb{K}$ and $d_{H^1}$ induce the same topology.
\end{fact}

 \subsection{Spaces of convex bodies}

Recall that $\mathcal{K}^n$ (resp. $\mathcal{K}^{n*}$) is the set of convex bodies in $\R^n$ (resp. convex bodies with positive intrinsic area). We denote by  $\mathcal{K}^n_S$  the space of convex bodies with Steiner point at the origin, and  $\mathcal{K}^{n*}_S=\mathcal{K}^n_S\cap \mathcal{K}^{n*}$. 

In the sequel,  a star as upper-script mean that we consider only convex bodies with positive intrinsic area (that is, we exclude points and segments). In the following table, it is obvious that all the sets without a star are in bijection, as well as all the sets with a star.

\begin{center}
\begin{tabular}{|c|c|c|c|}
\hline
\textbf{convex bodies} & \textbf{up to positive} &  \textbf{with $V_2=1$} & \textbf{with $V_1=1$} \\
\textbf{in $\R^n$...} & \textbf{scaling} &  & \\
\hline
\textbf{up to translations} & $\oshape^n$ and $\oshape^{n*}$ & &\\
\hline	
\textbf{with Steiner point} & & $\mathcal{K}^n_{SV_2}$ and $\mathcal{K}^{n*}_{SV_2}$ & $\mathcal{K}^n_{SV_1}$ and $\mathcal{K}^{n*}_{SV_1}$ \\
\textbf{at the origin} & & & \\
\hline
\end{tabular}
\end{center}
We have 
\begin{center}
\begin{tabular}{ c c c}
$\supp(\mathcal{K}_S^{n*}) \subset \mathcal{C}_n$, & $\supp(\mathcal{K}^{n*}_{SV_2}) \subset \mathcal{H}^\infty_n$, & $\supp(\mathcal{K}^{n*}_{SV_1}) \subset \Kl_n$~. \\
\end{tabular}
\end{center}
Clearly, $\mathcal{K}^{n*}_{SV_2}$ (resp.   $\mathcal{K}^{n*}_{SV_1}$) is in bijection with $\oshape^{n*}$, and we denote by  $d_{SV_2}$ (resp. $d_{SV_1}$) the pull-back of $\d$ on  $\mathcal{K}^{n*}_{SV_2}$ (resp.   $\mathcal{K}^{n*}_{SV_1}$).
By construction, the map $\supp$ defines isometries
\begin{eqnarray*}
(\mathcal{K}^{n*}_{SV_2},d_{SV_2}) &\overset{\sim}\longrightarrow &(\supp(\mathcal{K}^{n*}_{SV_2}),d_\mathcal{H})~, \\
(\mathcal{K}^{n*}_{SV_1},d_{SV_1})& \overset{\sim}\longrightarrow &(\supp(\mathcal{K}^{n*}_{SV_1}),d_\mathbb{K})~,
\end{eqnarray*}
and as all these sets are isometric to $(\oshape^{n*},\d)$. We immediately obtain some parts of  Theorems~\ref{main} and \ref{main1b}:
 $(\oshape^{n*},\d)$ is a metric space, isometric to a convex subset of $\mathbb{H}^\infty_n$. In turn, it has curvature  $\leq -1$  and $\geq -1$, as this is clearly true for its isometric image in the hyperbolic space, and it is a uniquely geodesic metric space, as the hyperbolic space is uniquely geodesic. The unique shortest path is the convex combination,  as the property occurs in $\Kl_n$.

Let us check two easy facts that give other parts of Theorems~\ref{main} and \ref{main1b}. The first one implies that 
 $\supp(\mathcal{K}^{n*}_{SH})$  is unbounded.
\begin{fact} $\supp(\mathcal{K}^{n*}_{SH})$ contains an entire geodesic of $\mathbb{H}_n^\infty$.\end{fact}
\begin{proof} In the plane, consider the following segments: $K_1=[-1,1]\times\{0\}$ and $K_2=\{0\} \times [-1,1]$. For $0\leq t \leq 1$, the convex combination $(1-t)K_1+t K_2$ is the rectangle $[-(1-t),1-t] \times [-t,t]$, whose Steiner point is $0$. This gives an entire geodesic of $\mathbb{H}_2^\infty$ contained in $\supp(\mathcal{K}^{2*}_{SH})$. \end{proof}

The following fact implies that $(\oshape^{n*},\d)$ has infinite Hausdorff dimension.

\begin{fact}
For any $s \in \N$, there is an open ball of the finite dimensional hyperbolic space $\mathbb{H}^s$ that isometrically embeds into  $(\oshape^{n*},\d)$.
\end{fact}
\begin{proof}
The convex hyperbolic polyhedra  constructed  
in \cite{BG} parametrize the similarity classes of convex polygons with fixed angles; by construction, they isometrically
embed into $(\oshape^{n*},\d)$. 
The dimension of the hyperbolic polyhedra is $(s-3)$ if the polygons have $s$ edges.  
\end{proof}

\begin{fact}
The boundary of $(\oshape^{n*},\d)$ is homeomorphic to the real projective space of dimension $(n-1)$
\end{fact}

\begin{proof}
The boundary is the space of segments, up to  homotheties: indeed, for example by looking at the isometric model $(\supp(\mathcal{K}_{SV_1}^{n*}), d_{\mathbb{K}})$, we see that the convex bodies $K$ on the boundary are the one for which $V_2(K)=0$ (see Fact \ref{lem bord infini}) and $V_1(K)=1$, and these are exactly unit length segments. Hence $\partial \oshape^{n*}$ is in bijection with $P^{n-1}(\R)$, the real projective space of dimension $n-1$ (that is, the space of lines in $\R^n$).

We can endow $\partial \oshape^{n*}$ with the visibility metric from $[B^n]$: the distance between $a,b \in \partial \oshape^{n*}$, denoted by $<_B(a,b)$,   is the angle (with value in $[0,\pi]$) between the two lines $c_a$ and $c_b$ from $[B^n]$ and with endpoints $a$ and $b$ respectively. But clearly,
the element of $O(n)$ sending the line $a$ to the line $b$  is also a $\d$-isometry sending $c_a$ to $c_b$. In turn, $\partial \oshape^{n*}$ endowed with the visibility metric is isometric to $P^{n-1}(\R)$ endowed with its round metric. 
From \cite[Proposition II.9.2]{BH}, $<_B:\partial \oshape^{n*} \times \partial \oshape^{n*} \to \R$ is continuous for the classical topology on $\partial \oshape^{n*}$. Hence for this topology, $\partial \oshape^{n*}$ is homeomorphic to $P^{n-1}(\R)$.

\end{proof}

In the two following sections we will prove the two remaining parts of Theorems~\ref{main} and \ref{main1b}: the assertion about terminal points of segments, and the topological properties.

\subsection{Terminal points of segments}\label{sec terminal}

Let $K_1,K_2 \in \mathcal{K}_{SV_1}^{n}$. The \emph{segment} between 
$K_1$ and $K_2$ is $\{(1-t) K_1+ t K_2, t \in [0,1]\}$.
We say that  $K_1\in \mathcal{K}_{SV_1}^{n}$ is a \emph{terminal point} of the segment 
if for any $t<0$, $(1-t)\supp(K_1)+t\supp(K_2) \notin \supp( \mathcal{K}_{SV_1}^{n})$. 
An \emph{extreme point} $K$ of  $\mathcal{K}_{SV_1}^{n}$ is 
such that there does not exist $K_1,K_2\in  \mathcal{K}_{SV_1}^{n}$, $K_1\not=K_2$, and $t\in(0,1)$ such that $\supp(K)=(1-t)\supp(K_1)+t\supp(K_2)$. 
In the plane, extreme points of $ \mathcal{K}_{SV_1}^{2}$ are
segments and triangles \cite[Theorem 3.2.14]{schneider}.
For $n \geq 3$, extreme points of  $ \mathcal{K}_{SV_1}^{n}$ are dense for the Hausdorff distance \cite[3.2.18]{schneider}. 

Clearly, an extreme point is a terminal point for all the segments ending at this point. But there are much more terminal points. For example, one can find convex bodies with a non smooth point on the boundary (i.e. a point of the convex body contained in more than one support plane) which are terminal points for the segment 
starting at the unit ball ---this idea is illustrated in Figure~\ref{fig:extreme}. 

\begin{figure}
\begin{center}
\psfrag{K}{$K$}
\psfrag{e}{$\epsilon$}
\psfrag{K+}{$K+\epsilon B^2$}
\psfrag{K-}{$K-\epsilon B^2$}
\includegraphics[width=0.4\linewidth]{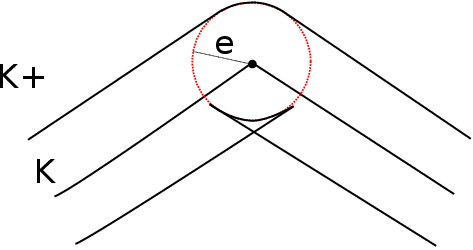}\caption{
If a plane convex body $K$ has a non-smooth point, then for any $\epsilon >0$, $\supp(K)+\epsilon \supp(B^2)$ is the support function of a convex body, while $\supp(K)-\epsilon \supp(B^2)$ is not.}\label{fig:extreme}
\end{center}
\end{figure}

In this section, we will use a different argument to prove
 that \emph{any} convex body is the terminal point of some segment (Proposition \ref{propositionpointsterminaux}), see Figure~\ref{fig:terminaux} for an example.

\begin{figure}
\begin{center}
\includegraphics[width=0.8\linewidth]{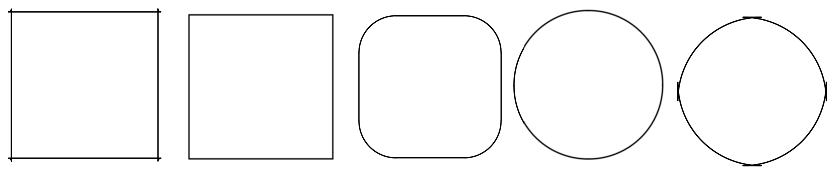}\caption{The disc and the square
are both terminal points of the segment joining them.}\label{fig:terminaux}
\end{center}
\end{figure}

If a function $h\in \Kl_n$ belongs to $\supp(\mathcal{K}_{SV_1}^{n*})$, then 
 its one-homogeneous extension  $\tilde{h}$ is convex, hence has non-negative Laplacian in the weak sense.  This means that for every non-negative function $\varphi \in C^\infty_c (\R^n)$, we have
\[
\int_{\R^n} \tilde{h}(x) \Delta_e \varphi(x)\mbox{d}x \geq 0~,
\]
where $C^\infty_c (\R^n)$ is the set of smooth functions with compact support in $\R^n$.

For $1 \leq p < n$, we will denote by $B_{p,n}$ the $p$-dimensional ball with radius $r_1(p)$ in $\R^n$, which is the set of points $x\in \R^n$  with $x_1^2 +\cdots+ x_p^2 \leq r_1(p)^2$ and $x_{p+1}= \cdots=x_n=0$. 
The number $r_1(p)$ is such that a ball with such radius has $V_1=1$.
We have $V_1(B_{p,n})=1$, hence $B_{p,n} \in \mathcal{K}_{SV_1}^{n}$ (note that $B_{p,n} \in \mathcal{K}_{SV_1}^{n*}$ if and only if $p \geq 2$).
 Let $b_{p,n} = \supp(B_{p,n}) \in \supp(\mathcal{K}_{SV_1}^{n})$ and 
let $\widetilde{b_{p,n}}(x)=r_1(p) \sqrt{x_1^2+\cdots+x_p^2}$ be the 1-homogeneous extension of $b_{p,n}$ (if $p=1$, then $\widetilde{b_p}(x)=r_1(1)|x_1|=\frac{|x_1|}{2}$). 

%


\begin{proposition} \label{propositionpointsterminaux}
Let $p \in \mathbb{N}$ such that $1 \leq p < n$. Then any $K \in \mathcal{K}_{SV_1}^{n*}$ is the terminal point of a segment in $\mathcal{K}_{SV_1}^{n*}$, which starts at some embedded $p-$dimensional ball in $\R^n$.
\end{proposition}

Actually the proof will show that there are infinitely many such segments.
If $p=1$, this ball is in fact a segment and lies on the boundary of $\Kl_n$.

To prove Proposition \ref{propositionpointsterminaux}, we need the following theorem due to Alexandrov (see \cite{hessian}).
\begin{theorem}\label{alex thm}
A convex function $f : \R^n \rightarrow \R$ is \emph{twice differentiable} at almost every $\bar{x} \in \R^n$, which means that for almost every $\bar{x} \in \R^n$, there exists a quadratic polynomial $Q_{\bar{x}}$, and a function $R_{\bar{x}}$, such that
\[
f(x) = Q_{\bar{x}}(x) + R_{\bar{x}}(x) \mbox{ and } \lim_{u \rightarrow 0} \frac{R_{\bar{x}}(\bar{x}+u)}{\|u\|^2}=0~.
\]
\end{theorem}

\begin{proof}[Proof of Proposition \ref{propositionpointsterminaux}]
Let $k = \supp(K) \in \supp(\mathcal{K}_{SV_1}^{n*})$, and let $\tilde{k}$ be its 1-homogeneous extension. Let $\bar{x} \in \R^n$ be a point at which $\tilde{k}$ is twice differentiable, and let $Q_{\bar{x}}$ and $R_{\bar{x}}$ be as in Theorem~\ref{alex thm}. Since $n >p$, the vector space $\{x_1=\cdots=x_p=0\}$ has positive dimension, hence, up to a rotation of $K$, we may assume that the first components of $\bar{x}$ are $\bar{x}_1=\cdots=\bar{x}_p=0$.

Let $\varphi \in C^\infty_c (\R^n)$ be a non-negative function, with support in the unit ball in $\R^n$, positive in a neighborhood of $0$, and with $\int_{\R^n} \varphi=1$. For $\epsilon>0$, let $\varphi_\epsilon \in C^\infty_c (\R^n)$ be the function $\varphi_\epsilon(x)=\frac{1}{\epsilon^n} \varphi(\frac{x-\bar{x}}{\epsilon})$: this function is non-negative, has support in $B(\epsilon,\bar{x})$ (the ball centered at $\bar{x}$ and with radius $\epsilon$), and $\int_{\R^n}\varphi_\epsilon=1$.

Let $t <0$. We want to show that $(1-t)k + t b_{p,n} \notin \supp(\mathcal{K}_{SV_1}^{n*})$. We argue by contradiction: assume that $(1-t)k + t b_{p,n} \in \supp(\mathcal{K}_{SV_1}^{n*})$. Then $(1-t)\tilde{k} + t \widetilde{b_{p,n}}$ is a convex function on $\R^n$, hence its Laplacian is non-negative in the weak sense, so in particular we have
\begin{equation} \label{eqdemopointterminal}
\int_{\R^n}((1-t)\tilde{k} + t \widetilde{b_{p,n}}) \Delta_e \varphi_\epsilon 	 \geq 0~.
\end{equation}
 We will first show that we always have
\begin{equation} \label{eqdemopointterminalbis}
\int_{\R^n}\tilde{k} \Delta_e \varphi_\epsilon \underset{\epsilon \rightarrow 0} \longrightarrow + \infty ~.
\end{equation}
Since $t$ is negative, with equation \eqref{eqdemopointterminal} it is sufficient to show that
\begin{equation} \label{eqdemopointterminalter}
\int_{\R^n}\widetilde{b_{p,n}} \Delta_e \varphi_\epsilon \underset{\epsilon \rightarrow 0} \longrightarrow + \infty ~.
\end{equation}
Now we need to argue depending whether $p=1$ or $p \geq 2$.
\begin{itemize}[nolistsep]
\item If $p \geq 2$ we have $\Delta_e \widetilde{b_{p,n}} (x) = \frac{r_1(p)(p-1)}{\sqrt{x_1^2 + \cdots+x_p^2}}$, and since $\bar{x}_1=\cdots=\bar{x}_p=0$ we have $\sqrt{x_1^2+\cdots+x_p^2} \leq \|x-\bar{x}\|$, hence $\Delta_e \widetilde{b_{p,n}} (x) \geq \frac{r_1(p)(p-1)}{\epsilon}$ for every $x \in B(\epsilon,\bar{x})$, so we have (by  Green Formula)
\[
\int_{\R^n} \widetilde{b_{p,n}} \Delta_e \varphi_\epsilon =  \int_{B(\epsilon,\bar{x})} \varphi_\epsilon \Delta_e \widetilde{b_{p,n}} \geq \frac{r_1(p)(p-1)}{\epsilon} \int_{B(\epsilon,\bar{x})}  \varphi_\epsilon=\frac{r_1(p)(p-1)}{\epsilon}~,
\]
and this gives \eqref{eqdemopointterminalter}.

\item If $p=1$, then we have
\begin{eqnarray*}
\int_{\R^n} \widetilde{b_{p,n}}(x) \Delta_e \varphi_\epsilon(x) \mbox{d}x &=& \frac{1}{2} \int_{\R^n} |x_1| \Delta_e \varphi_\epsilon(x) \mbox{d}x\\
 &=&  \int_{\R^{n-1}} \varphi_\epsilon (0,x_2,\ldots,x_n) \mbox{d}x_2 \ldots \mbox{d}x_n \\
&=& \frac{1}{\epsilon^n} \int_{\R^{n-1}} \varphi \left(0,\frac{x_2-\bar{x}_2}{\epsilon},\ldots,\frac{x_n-\bar{x}_n}{\epsilon}\right) \mbox{d}x_2 \ldots \mbox{d}x_n \\
&=& \frac{1}{\epsilon} \int_{\R^{n-1}} \varphi (0,y_2,\ldots ,y_n) \mbox{d}y_2 \ldots \mbox{d}y_n ~.
\end{eqnarray*}
The second equality is a classical computation, the third is true because $\bar{x}_1=0$, and for the last one we use the change of variable $y_i = \frac{x_i - \bar{x}_i}{\epsilon}$. Since $\varphi$ is positive in a neighborhood of zero, we have $\int_{\R^{n-1}} \varphi (0,y_2,\ldots,y_n) \mbox{d}y_2 \ldots \mbox{d}y_n >0$, and this gives (\ref{eqdemopointterminalter}).
\end{itemize}

Moreover, since $\widetilde{k}=Q_{\bar{x}} + R_{\bar{x}}$, we have
\[
\int_{\R^n}\widetilde{k} \Delta_e \varphi_\epsilon = \int_{\R^n}Q_{\bar{x}} \Delta_e \varphi_\epsilon + \int_{\R^n} R_{\bar{x}}\Delta_e \varphi_\epsilon~.
\]
The function $Q_{\bar{x}}$ is a quadratic polynomial, hence its Laplacian is equal to a constant $C \in \R$, which gives $\int_{\R^n}Q_{\bar{x}} \Delta_e \varphi_\epsilon = \int_{\R^n} C \varphi_\epsilon =C$. And since $\Delta_e \varphi_\epsilon(x) =\frac{1}{\epsilon^{n+2}} \Delta_e \varphi(\frac{x-\bar{x}}{\epsilon})$, with the change of variable $y=\frac{x-\bar{x}}{\epsilon}$, we have
\begin{eqnarray*}
\int_{\R^n} R_{\bar{x}}(x) \Delta_e \varphi_\epsilon(x) \mbox{d}x &=& \frac{1}{\epsilon^{n+2}} \int_{B(\epsilon,\bar{x})} R_{\bar{x}}(x) \Delta_e \varphi \left(\frac{x-\bar{x}}{\epsilon} \right)\mbox{d}x \\
& =& \frac{1}{\epsilon^2} \int_{B(1,0)} R_{\bar{x}}(\bar{x}+\epsilon y) \Delta_e \varphi (y) \mbox{d}y ~.
\end{eqnarray*}
Since $\frac{R_{\bar{x}}(\bar{x}+u)}{\|u\|^2} \underset{u \rightarrow 0} \longrightarrow 0$, there exists $M>0$ such that $| R_{\bar{x}}(\bar{x}+u)| \leq M \|u\|^2$ for $\|u\|$ small enough, hence for $\epsilon$ small enough we have, for every $y \in B(1,0)$, $|R_{\bar{x}}(\bar{x}+ \epsilon y)| \leq M \epsilon^2 \|y\|^2$, hence we obtain
\[
|\int_{\R^n} R_{\bar{x}}(x) \Delta_e \varphi_\epsilon(x) \mbox{d}x | \leq M \int_{B(1,0)} \|y\|^2 |\Delta_e \varphi(y)| \mbox{d}y~.
\]
The integral $\int_{\R^n} R_{\bar{x}} \Delta_e \varphi_\epsilon$ does not go to $+\infty$ when $\epsilon$ goes to zero, and by (\ref{eqdemopointterminalbis}) this is a contradiction.
\end{proof}

\subsection{Comparison of topologies}\label{haus}

We want to compare the topologies given by $d_{\mathbb{K}}$ and $d_\infty$ on  $\supp(\mathcal{K}_{SV_1}^{n*})$, where $d_\infty$ is the distance given by the sup norm. As a tool, we will use the distances $d_{L^2}$ and $d_{H^1}$ induced by the $L^2$ and $H^1$ norms respectively on $\sob_1$, as well as the following theorem, see  \cite{vitale} and \cite[Proposition~2.3.1]{groemer}.

\begin{theorem}[{Vitale}]\label{vitale}
The  distances $d_{\infty}$ and $d_{L^2}$ induce the same topology on $\supp(\mathcal{K}^n) \subset C^0(\mathbb{S}^{n-1})$. 
\end{theorem}
The result is weaker than saying that the two norms are equivalent on the space of convex bodies, that is not true, see \cite{vitale} for details.

\begin{corollary} \label{cor equi L2 H1}
The  distances $d_\infty$, $d_{L^2}$ and $d_{H^1}$ induce the same topology on $\supp(\mathcal{K}^n)$. 
\end{corollary}
\begin{proof}
We prove that $d_{L^2}$ and $d_{H^1}$ induce the same topology.
If $h_i \rightarrow h$ for $\|\cdot\|_{H^1}$, then obviously $h_i \rightarrow h$ for $\|\cdot\|_{L^2}$. And if $h_i \rightarrow h$ for $\|\cdot\|_{L^2}$, then by Theorem \ref{vitale} we have $h_i \rightarrow h$ for $d_\infty$. Let us check that this implies the convergence for $d_{H^1}$. This is obvious that $h_i \rightarrow h$ in $L^2$. Moreover, let $R>0$ be such that $h_i \leq R$ for every $i$. Then $(\nabla h_i)_i$ almost everywhere converges pointwise to $\nabla h$, hence the convergence holds in $L^2$ via Lebesgue Dominated Convergence Theorem: these functions are uniformly bounded by $R$ as the $h_i$ are $R$-Lipschitz. Hence $h_i \rightarrow h$ for $\|\cdot\|_{H^1}$.
\end{proof}

A direct consequence of Fact \ref{prop equiv topo Klein} and Corollary \ref{cor equi L2 H1} is the following corollary, which relates the distances $d_\infty$ and $d_\mathbb{K}$.
\begin{proposition}\label{same cor}
On  $\supp(\mathcal{K}_{SV_1}^{n*})$, $d_\infty$ and $d_{\mathbb{K}}$ (as well as $d_{L^2}$ and $d_{H^1}$) induce the same topology.
\end{proposition}

As  $d_\infty$ clearly induces the same topology on $\supp(\mathcal{K}_{SV_1}^{n*})$ and $\supp(\mathcal{K}_{SV_2}^{n*})$, we obtain the last point of Theorem~\ref{main}, as the Hausdorff distance for convex bodies is exactly $d_\infty$ for the support functions.

\begin{remark}{\rm
	Even if $d_\infty$ and $d_\mathbb{K}$ induce the same topology, their behavior is quite different. First, similarly to the comparison between Euclidean and hyperbolic metric on the disc, $(\supp(\K^{n*}_{SV_1}),d_\infty)$ is bounded and $(\supp(\K^{n*}_{SV_1}),d_\mathbb{K})$ is not.  Also,  if segments are also shortest paths for 
 the Hausdorff distance, they are not unique in general, see note 11 of Section~1.8 in \cite{schneider}.
	}\end{remark}

Let us now check that  $(\supp(\mathcal{K}_{SV_1}^{n*}), d_{\mathbb{K}})$ is a proper metric space. It will be an immediate consequence of  Blaschke Selection Theorem together with Proposition~\ref{same cor}.
%

\begin{proposition}\label{prop proper}
$(\supp(\mathcal{K}_{SV_1}^{n*}), d_{\mathbb{K}})$ is a proper metric space.
\end{proposition}
\begin{proof}

Let $A$ be a closed bounded subset of $(\supp(\K^{n*}_{SV_1}),d_\mathbb{K})$. We want to show that $A$ is compact for $d_\mathbb{K}$; by Proposition~\ref{same cor}, it suffices to show that it is compact for $d_\infty$. As $(\supp(\mathcal{K}_{SV_1}^{n}),d_\infty)$ is compact (see p. 165 in \cite{schneider}), it suffices to show that $A$ is closed in $(\supp(\mathcal{K}_{SV_1}^{n}),d_\infty)$.

So assume $(h_i)_i$ is a sequence of elements of $A$ converging to  $h\in\supp(\K^{n}_{SV_1})$ for $d_\infty$; we want to show that $h \in A$. If $h \in \supp(\K^{n*}_{SV_1})$, then this is true, because Proposition~ \ref{same cor} implies that $A$ is a closed subset of $(\supp(\K^{n*}_{SV_1}),d_\infty)$. Otherwise, $h\in\supp(\K^{n}_{SV_1})\setminus  \supp(\K^{n*}_{SV_1})$, hence $\A(h)=0$ and it follows from Corollary~\ref{cor equi L2 H1} that $\A(h_i)\to 0$.
Then by Fact~\ref{lem bord infini}, the distance in $(\Kl_n,d_\mathbb{K})$ between $h_i$ and any given point $k \in \Kl_n$ goes to infinity, and that contradicts the fact that $A$ is a bounded subset of $(\supp(\K^{n*}_{SV_1}),d_\mathbb{K})$.
\end{proof}

Theorem~\ref{main} is now proved.

The two following facts conclude the proof of  Theorem~\ref{main1b}:
\begin{itemize}[nolistsep]
\item Since $(\oshape^{n*},\d)$ is proper, it is complete, hence $(\supp(\mathcal{K}^{n*}_{SH}),d_\mathbb{H})$ is also complete, so $\supp(\mathcal{K}^{n*}_{SH}) \subset \mathbb{H}_n^\infty$ is a closed subspace. 
\item Now, let us prove that $\supp(\mathcal{K}^{n*}_{SH})$ has empty interior. If this is not true, then there exists a ball $B$ in $(\mathbb{H}^\infty_n,d_\mathbb{H})$ such that $B \subset \supp(\mathcal{K}^{n*}_{SH})$; we can even assume that $\bar B$ (the closure of $B$) satisfies $\bar B \subset \supp(\mathcal{K}^{n*}_{SH})$. Since $(\supp(\mathcal{K}^{n*}_{SH}),d_\mathbb{H})$ is proper, closed balls are compact, hence $\bar B$ is compact. Hence there exists a non-empty relatively compact open set in $(\Kl_n,d_{\mathbb{K}})$. But that would be true for the infinite-dimensional Banach space $(\sob_{01},d_{01})$, and that is impossible: a closed ball would be compact.
\end{itemize}

\begin{figure}
\begin{center}
\psfrag{P}{$(\Vb)^{-1}(0)$}
\psfrag{a}{$(\A)^{-1}(0)$}
\psfrag{K}{$\supp(\mathcal{K}_S^{n*})$}
\psfrag{dK}{$\partial \K$}
\psfrag{+}{$\scriptsize{+}$}
\includegraphics[width=0.7\linewidth]{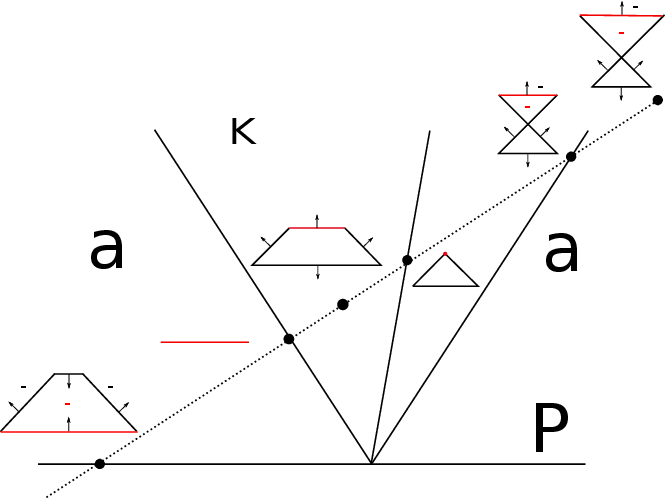}\caption{The minus sign outside of the polygons indicates edges with negative algebraic length, while the minus sign inside a polygon indicates a negative $\A$. }\label{fig:quadrangl}
\end{center}
\end{figure}

\begin{remark}\label{rem ref}{\rm
As far as we know, the idea associate a hyperbolic metric to spaces of convex
bodies \emph{via} the area form and support function was more or less explicit in the $90$'s, for spaces of convex polygones. The main reference is \cite{BG}, see \cite{EM} for detailed references. This construction was extended to spaces of convex polytopes in \cite{FI}. 

The smallest vector space containing $\supp(\K^{n})$  as a convex cone
is the vector space spanned by the cone:
$$\mathrm{Sonic}^n=\{ h-k | h,k\in  \supp(\K^{n})\}~, $$
the space of \emph{$n$-dimensional hedgehogs}. 
See \cite[9.6]{schneider}, \cite{schnei-h} and the references therein for more information. 
Let us say that 
the name was coined in \cite{LLR1}, although
they previously appeared in the literature  under different names, see \cite{oliker}.
If $h\in \mathrm{Sonic}^n$, there is a way to associate a geometric object in $\R^n$, see  \cite{schnei-h,MM}, that is illustrated in most of the figures of the present article. 
A description of $\mathrm{Sonic}^2$ in $C^0(\mathbb{S}^1)$ is contained in \cite{MM}.  
But $\mathrm{Sonic}^n$ is not complete for any reasonable norm on it
---it contains  
 $C^2(\mathbb{S}^{n-1})$, so it is dense in both $\sob$ and $C^0(\S^{n-1})$ endowed with their classical norms. 
Particular cases of the results of the present article were achieved in this setting (mostly in the regular case) in \cite{mmarx0,mmarx,mmarx2}.
}\end{remark}

\section{The space of shapes $\shape^{n*}$}\label{sec shapes}

\subsection{Immediate properties}

Let $\shape^{n*}$ be the quotient of $\oshape^{n*}$ by linear isometries of the Euclidean space $\R^n$: the action of $O(n)$ on 
$\oshape^{n*}$ is defined by $\Phi[K]:=[\Phi K]$. For $K\in\mathcal{K}^{n*}$,  we will denote by $\llbracket K\rrbracket$
the set of convex bodies differing from $K$ by positive scaling and Euclidean isometries. 

Since $\sa$ is $O(n)-$invariant, we have $\d(\Phi[K_1],\Phi[K_2])=\d([K_1],[K_2])$, so $O(n)$ acts by isometries on $\oshape^{n*}$.
Moreover, the action of $O(n)$ is clearly continuous on support functions for $d_\infty$, hence by Proposition~\ref{same cor}, 
the action is continuous on  $(\shape^{n*},\d)$. 
Let us introduce
\begin{equation}\label{def bd}\bd(\llbracket K_1\rrbracket,\llbracket K_2\rrbracket)=\inf_{\Phi,\Phi'\in O(n)} \d (\Phi[K_1],\Phi'[K_2])~.\end{equation}
Noting that by continuity and compactness, the infimum is actually a minimum,
 it is not hard to deduce that 
$\bd$ is a distance.

\begin{proposition} \label{prop:shapeCBBproper}
$(\shape^{n*},\bd)$ is a proper geodesic metric space with curvature $\geq -1$.
\end{proposition}
\begin{proof}
It is a general fact that the quotient will be geodesic and with curvature $\geq -1$, see for example Proposition~10.2.4 in \cite{BBI}. The fact that the quotient is proper is also very general. Indeed, suppose that  $(\llbracket K_i\rrbracket)_{i\in \N}$ is a bounded sequence in $(\shape^{n*},\bd)$.  There are  $\Phi_i \in O(n)$ such that $(\Phi_i[K_i])_{i\in \N}$
is a bounded sequence in $(\oshape^{n*},\d)$. Since $(\oshape^{n*},\d)$ is proper, up to extract a subsequence, there exists $[K] \in \oshape^{n*}$ such that $\d(\Phi_i[K_i],[K])\to 0$.
As $\bd(\llbracket K_i\rrbracket, \llbracket K\rrbracket) \leq \d(\Phi_i[K_i],[K])$, we have $\bd(\llbracket K_i\rrbracket, \llbracket K\rrbracket) \to 0$.
\end{proof}

\subsection{Non-uniqueness of shortest paths in $\shape^{n*}$}\label{sectionnonuniqueness}

The aim of this section is to prove that shortest paths are not unique in $\shape^{n*}$. Obviously, since $\shape^{2*}$ isometrically embedds into $\shape^{n*}$ for $n \geq 2$, it is sufficient to prove this property for $n=2$. Hence, in this section, we consider convex bodies in $\R^2$. We will produce a handmade example.

Let $K$ be the intersection of the half-space $[0,\infty) \times \R$ with the ellipse with center 0, width $2\sqrt{2}$ and height $\frac{2}{\sqrt{2}}$. The support function of $K$ is a function on $\mathbb{S}^1$, and with the parametrization $x=(\cos s,\sin s) \in \mathbb{S}^1$, for $s\in[0,2\pi]$, we will actually define the 
support function $k$ of $K$ on $[0,2\pi]$. Namely,
$$
k(s) = \sqrt{2 \cos^2 s + \frac{1}{2} \sin^2 s} \mbox{ for }  s \in [-\frac{\pi}{2}, \frac{\pi}{2}], \mbox{ and } k(s) = \frac{1}{\sqrt{2}} |\sin s| \mbox{ for } s \in [\frac{\pi}{2}, \frac{3\pi}{2}]~.
$$
Let $(\beta,0)$ be the Steiner point of $K$, and let $\alpha = V_1(K) = \frac{1}{2} \int_0^{2\pi} k \simeq 2.4$. Then the convex body $K_1 = \alpha^{-1}K + (-\alpha^{-1}\beta,0)$ has Steiner point 0, and $V_1(K_1)=1$: hence $K_1 \in\mathcal{K}_{SV_1}^{2*}$. Its support function $k_1 \in \supp(\mathcal{K}_{SV_1}^{2*})$ is given by
$$ k_1(s) = \alpha^{-1} \left(\sqrt{2 \cos^2 s + \frac{1}{2} \sin^2 s}-\beta \cos s \right) \mbox{ for } s \in [\frac{-\pi}{2},\frac{\pi}{2}]$$
and
$$ k_1(s) = \alpha^{-1}\left( \frac{1}{\sqrt{2}}|\sin s| - \beta \cos s \right) \mbox{ for } s \in [\frac{\pi}{2},\frac{3\pi}{2}]~.$$

Let $K_2$ be the rectangle $[-\frac{2}{5},\frac{2}{5}]\times [-\frac{1}{10},\frac{1}{10}]$. Obviously, 0 is the Steiner point of $K_2$. Its support function is defined for any $s \in [0,2\pi]$ by
$$ k_2(s) = \frac{2}{5} |\cos s| + \frac{1}{10} |\sin s|~,$$
and since $K_2 = [-\frac{2}{5},\frac{2}{5}] \times \{0\} +  \{0\} \times [-\frac{1}{10},\frac{1}{10}]$, we have $V_1(K_2) = \operatorname{length}([-\frac{2}{5},\frac{2}{5}]) + \operatorname{length}([-\frac{1}{10},\frac{1}{10}])=1$. Hence $K_2 \in \mathcal{K}_{SV_1}^{2*}$ and $k_2 \in \supp (\mathcal{K}_{SV_1}^{2*})$.

Let $\llbracket K_1 \rrbracket$ and $\llbracket K_2 \rrbracket$  be the corresponding equivalent classes in $\shape^{2*}$. Since $K_2$ is invariant by the symmetry with respect to the horizontal line, the distance between $\llbracket K_1 \rrbracket$ and $\llbracket K_2 \rrbracket$ is given by
$$
\bdd (\llbracket K_1 \rrbracket, \llbracket K_2 \rrbracket) = \min_{\theta \in \R} \dd ([K_1], R_\theta [K_2])~,
$$
where we denote by $R_\theta$ the rotation of angle $\theta$ in $\R^2$. We will prove the following:
\begin{proposition} \label{prop non unicite}
The minimum is obtained for $\theta=0$ and $\theta=\frac{\pi}{2}$, that is we have
$$\bdd (\llbracket K_1 \rrbracket, \llbracket K_2 \rrbracket) = \dd ([K_1], [K_2]) =\dd ([K_1], R_{\frac{\pi}{2}}[K_2]) ~.$$
\end{proposition}

Let us state the following fact.  Note that in general, this is not true that every shortest path in a quotient space is obtained as the projection of a shortest path.

\begin{lemma}\label{lem geod shape}
Let $[K_1], [K_2]\in \oshape^{n*}$, and let $\Phi \in O(n)$ be such that $\bd(\llbracket K_1 \rrbracket, \llbracket K_2 \rrbracket)=\d([K_1],\Phi[K_2])$. Suppose that $[\gamma]$ is the shortest path between $[K_1]$ and $\Phi[K_2]$. Then the projection 
$\llbracket \gamma \rrbracket$ is a shortest path between $\llbracket K_1 \rrbracket$ and  $\llbracket K_2 \rrbracket$. Moreover, the projection is an isometry from $[\gamma]$ to $\llbracket \gamma \rrbracket$.
\end{lemma}
 
\begin{proof}
Let us suppose that $[\gamma]:[0,1]\to X$ is  affinely parametrized. Then, for any $0\leq s\leq t \leq 1$,
\[
 \bd( \llbracket \gamma(s)\rrbracket, \llbracket \gamma(t)\rrbracket) \leq \d([\gamma(s)],[\gamma(t)])= (t-s) \d([K_1],\Phi[K_2]) =(t-s) \bd(\llbracket K_1 \rrbracket, \llbracket K_2 \rrbracket)~.
\]
Using three times this inequality, we obtain
\begin{eqnarray*}
\bd(\llbracket K_1 \rrbracket, \llbracket K_2 \rrbracket) & \leq & \bd(\llbracket \gamma(0)\rrbracket, \llbracket \gamma(s)\rrbracket) + \bd(\llbracket \gamma(s)\rrbracket, \llbracket \gamma(t)\rrbracket) + \bd(\llbracket \gamma(t)\rrbracket, \llbracket \gamma(1))\rrbracket \\
 & \leq & (s+(t-s)+(1-t)) \bd(\llbracket x\rrbracket, \llbracket y\rrbracket) = \bd(\llbracket x\rrbracket, \llbracket y\rrbracket)~.
\end{eqnarray*}
All these inequalities are equalities, so in particular $$\bd(\llbracket\gamma(s)\rrbracket, \llbracket \gamma(t)\rrbracket) = (t-s)\bd(\llbracket K_1 \rrbracket, \llbracket K_2 \rrbracket)~.$$
\end{proof}

Proposition~\ref{prop non unicite} is sufficient to prove the non-uniqueness of shortest paths in $\shape^{2*}$. Indeed, Lemma~\ref{lem geod shape} shows that the projections of the shortest paths in $\oshape^{2*}$ between $[K_1]$ and $[K_2]$, and between $[K_1]$ and $R_{\frac{\pi}{2}}[K_2]$, are again shortest paths in $\shape^{2*}$. But these two shortest paths are different: the first shortest path contains the point $\llbracket \frac{1}{2} K_1 + \frac{1}{2}K_2 \rrbracket$, and this point is not on the second shortest path $t \mapsto \llbracket (1-t)K_1 + t R_{\frac{\pi}{2}} (K_2) \rrbracket$:
$\frac{1}{2}K_1 + \frac{1}{2}K_2$ is not the image by a rotation of $(1-t)K_1 + t R_{\frac{\pi}{2}}(K_2)$, which is equivalent to say that $\frac{1}{2 }\alpha^{-1} K + \frac{1}{2}K_2$ is not the image by a rotation and a translation of 
$(1-t)\alpha^{-1} K + t R_{\frac{\pi}{2}}(K_2)$. See Figure \ref{fig:ellipse}.

\begin{figure}
\begin{center}
\includegraphics[scale=0.6]{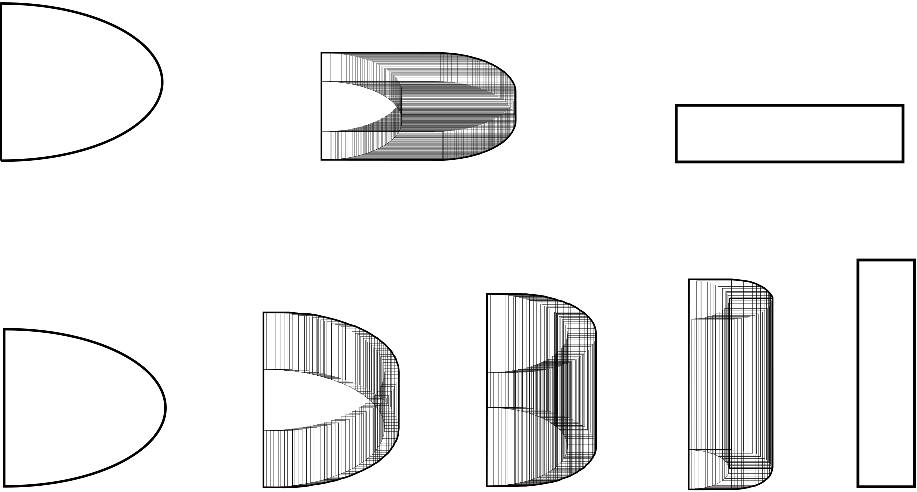}\caption{The convex body
$\frac{1}{2 }\alpha^{-1} K + \frac{1}{2}K_2$ (middle of the upper line) is not the image by a rotation and a translation of
$(1-t)\alpha^{-1} K + t R_{\frac{\pi}{2}}(K_2)$ (represented on the bottom line for $t=0,\frac{1}{4},\frac{1}{2},\frac{3}{4},1$).}\label{fig:ellipse}
\end{center}
\end{figure}

Since $R_\pi [K_2] = [K_2]$, to compute the minimum this is sufficient to consider $\theta \in [-\frac{\pi}{2},\frac{\pi}{2}]$. Moreover, let $T$ be the symmetry with respect to the $x$ axis: we have $T[K_1]=[K_1]$, hence we have
$$\dd ([K_1], R_\theta [K_2])=\dd (T[K_1], R_\theta [K_2])=\dd ([K_1], T \circ R_\theta [K_2]) = \dd ([K_1], R_{-\theta} [K_2])~.$$
This shows that in fact we need only to consider $\theta \in [0,\frac{\pi}{2}]$.

Let $k_2^\theta$ be the support function of $R_\theta[K_2]$, that is $k_2^\theta(s) = k_2(s-\theta)$. We have
$$
\cosh(\dd ([K_1], R_\theta [K_2])) = \frac{V_2(k_1,k_2^\theta)}{\sqrt{V_2(k_1)V_2(k_2^\theta)}} = \frac{f(\theta)}{2\sqrt{V_2(k_1)V_2(k_2)}}~,
$$
where we denote by $f(\theta)$ the function defined by
$$
f(\theta) = \int_0^{2\pi} (k_1(s)k_2(s-\theta) - k_1'(s)k_2'(s-\theta)) \mbox{d}s~.
$$
Proposition \ref{prop non unicite} is a direct consequence of the following lemma.
\begin{lemma}
On $[0,\frac{\pi}{2}]$, $f$ attains its minimum at the points $\theta=0$ and $\theta=\frac{\pi}{2}$.
\end{lemma}

\begin{proof}
Fix $\theta \in (0,\frac{\pi}{2})$, and consider the function $s \mapsto k_1(s) k_2'(s-\theta)$. This function is piecewise $\mathcal{C}^1$, but is not continuous: the function $k_2'(s-\theta)$ has jumps, with height $\frac{1}{5}$ at the points $s= \theta$ and $s= \pi+\theta$,  and with height $\frac{4}{5}$ at the points $s= \frac{\pi}{2} + \theta$ and $s= \frac{3\pi}{2}+\theta$. Hence we have

$$\int_0^{2\pi} (k_1(s)k_2'(s-\theta))' \mbox{d}s = -\frac{1}{5} k_1(\theta) - \frac{1}{5} k_1(\pi+\theta) - \frac{4}{5}k_1(\frac{\pi}{2} + \theta) - \frac{4}{5} k_1(\frac{3\pi}{2} + \theta)
$$
$$ = -\frac{1}{5\alpha} \sqrt{2 \cos^2 \theta + \frac{1}{2}\sin^2 \theta} -\frac{4}{5\alpha} \sqrt{2 \sin^2 \theta + \frac{1}{2}\cos^2 \theta} - \frac{1}{5 \sqrt{2}\alpha} \sin \theta  - \frac{4}{5\sqrt{2}\alpha} \cos\theta~.
$$
The equality $(k_1 k'_2)'=k_1' k_2' +k_1 k_2''$ gives $-k_1' k_2' = k_1 k_2'' - (k_1k_2')'$, so
$$- \int_0^{2\pi} k_1'(s)k_2'(s-\theta) \mbox{d}s = \int_0^{2\pi}  (k_1(s)k_2''(s-\theta) -(k_1(s)k_2'(s-\theta))')\mbox{d}s~,$$
and since $k_2(s-\theta)+ k_2''(s-\theta)=0$ for almost every $s \in [0,2\pi]$ we finally obtain
\begin{eqnarray*}
f(\theta) &=& \int_0^{2\pi} (k_1(s)k_2(s-\theta) - k_1'(s)k_2'(s-\theta)) \mbox{d}s \\
&=& \int_0^{2\pi} (k_1(s)(k_2(s-\theta)+k_2''(s-\theta)) - (k_1(s)k_2'(s-\theta))')\mbox{d}s \\
&=& \frac{1}{5\alpha} \sqrt{2 \cos^2 \theta + \frac{1}{2}\sin^2 \theta} +\frac{4}{5\alpha} \sqrt{2 \sin^2 \theta + \frac{1}{2}\cos^2 \theta} + \frac{1}{5 \sqrt{2}\alpha} \sin \theta  + \frac{4}{5\sqrt{2}\alpha} \cos\theta~.
\end{eqnarray*}
We easily check that $f(0)=f(\frac{\pi}{2})=\frac{\sqrt{2}}{\alpha}$ (the parameters of the ellipse and the segment have been chosen so that this property holds). And a direct computation shows that $f'(0)=\frac{1}{5\sqrt{2}\alpha}>0$ and $f'(\frac{\pi}{2}) = -\frac{4}{5\sqrt{2}\alpha} <0$. Moreover, let $g : [0,1] \rightarrow [0,\infty)$ be defined by
$$
g(u)= \frac{1}{5\alpha} \sqrt{\frac{3}{2} u + \frac{1}{2}} +\frac{4}{5\alpha} \sqrt{2 - \frac{3}{2} u } + \frac{1}{5 \sqrt{2}\alpha} \sqrt{1-u}  + \frac{4}{5\sqrt{2}\alpha} \sqrt{u}~.
$$
With the identity $\cos^2+\sin^2=1$, we easily check that $g(\cos^2 \theta)=f(\theta)$ for any $\theta \in [0,\frac{\pi}{2}]$. Hence $f'(\theta)=-2 g'(\cos^2 \theta)\sin \theta \cos \theta$. But $g$ is strictly concave, hence $g'$ has at most one zero on $[0,1]$, hence $f'$ has also at most one zero on $(0,\frac{\pi}{2})$. And this ends the proof: if the minimum of $f$ on $[0,\frac{\pi}{2}]$ was attained at a point $\theta \notin \{0,\frac{\pi}{2}\}$, since $f'(0)>0$ and $f'(\frac{\pi}{2})<0$, $f'$ would have at least 3 zeros on $(0,\frac{\pi}{2})$, and that is impossible.
\end{proof}

\subsection{Embedding of hyperbolic planes}\label{sectionembeddinghyp}

Trivially, for any $\Phi\in O(n)$ we have $\Phi [B^n]=[B^n]$. Apart from the fact that the action of $O(n)$ on $\oshape^{n*}$ is not proper, this says that for any $[K]\in \oshape^{n*}$, 
\begin{equation}\label{eq:metrique boule}\bd(\llbracket K \rrbracket, \llbracket B^n \rrbracket)=\d([K],[B^n])~.\end{equation}
From this we  first deduce the following fact.

\begin{fact}[Uniqueness of shortest paths starting from $B^n$]  Let $\llbracket K \rrbracket \in \shape^{n*}$. Then there is a \emph{unique} shortest path from $\llbracket  B^n \rrbracket$ to $\llbracket K \rrbracket$, which is the projection of the shortest path in $\oshape^{n*}$ between $[B^n]$ and $[K]$.
\end{fact}
\begin{proof}
Let $\bar \delta : [0,\bd(\llbracket B^n \rrbracket, \llbracket K \rrbracket)] \rightarrow \shape^{n*}$ be an arc-length parametrized shortest path between $\llbracket B^n \rrbracket$ and $\llbracket K \rrbracket$, and let $[\delta(t)] \in \oshape^{n*}$ be such that $\bar \delta(t) = \llbracket \delta(t) \rrbracket$. Let $t \mapsto [\gamma(t)]$ be the (unique) arc-length parametrized shortest path in $\oshape^{n*}$ between $[B^n]$ and $[K]$: we want to show that $\llbracket \delta(t) \rrbracket = \llbracket \gamma(t) \rrbracket$.

For any $t \in [0,\bd(\llbracket B^n \rrbracket, \llbracket K \rrbracket)]$, let $\Phi_t \in O(n)$ be such that $$\bd (\llbracket K \rrbracket ,\llbracket \delta(t) \rrbracket)=\d([K], \Phi_t[\delta(t)])~.$$
Since $t \mapsto \llbracket \delta(t) \rrbracket$ is a geodesic in $\shape^{n*}$, we have
\begin{eqnarray*}
 \d([B^n],\Phi_t[\delta(t)]) + \d(\Phi_t[\delta(t)],[K]) &=& \bd (\llbracket B^n \rrbracket ,\llbracket \delta(t) \rrbracket) + \bd (\llbracket \delta(t) \rrbracket ,\llbracket K \rrbracket) \\
& = & \bd (\llbracket B^n \rrbracket ,\llbracket K \rrbracket) = \d ([B^n],[K])~.
\end{eqnarray*}
Hence $\Phi_t[\delta(t)]$ is on the shortest path between $[B^n]$ and $[K]$ in $\oshape^{n*}$. Moreover, we have $\d ([B^n], \Phi_t[\delta(t)]) = \bd(\llbracket B^n \rrbracket, \llbracket \delta(t) \rrbracket)=t$ (the geodesic $t \mapsto \llbracket \delta(t) \rrbracket$ is arc-length parametrized), so $\Phi_t[\delta(t)] = [\gamma(t)]$ (remember that the geodesic
$t \mapsto [\gamma(t)]$ is also arc-length parametrized). Finally this gives 
$\llbracket \delta(t) \rrbracket= \llbracket \gamma(t) \rrbracket$.
\end{proof}

In turn, we can construct totally geodesic hyperbolic surfaces in $\shape^{n*}$. Interestingly,  many properties in this section are very general, but this one uses Alexandrov--Fenchel Inequality.

\begin{proposition}
Let $\llbracket P \rrbracket, \llbracket Q \rrbracket \in \shape^{n*}$ be such that 
$\llbracket P \rrbracket, \llbracket Q \rrbracket$ and $\llbracket B^n \rrbracket$ are three different points. Let $A \in O(n)$ be such that $\bd (\llbracket P \rrbracket, \llbracket Q \rrbracket)=\d([P],A[Q])$. Then the projection $\oshape^{n*} \rightarrow \shape^{n*}$, when restricted to the (plain) geodesic triangle with vertices $[B^n], [P]$ and $A[Q]$, is an isometry onto its image.
\end{proposition}

\begin{proof}
Without loss of generality, we may assume that $A$ is the identity (that is, $\bd (\llbracket P \rrbracket, \llbracket Q \rrbracket)=\d([P],[Q])$). Let $[K_1]$ and $[K_2]$ be in the geodesic triangle with vertices $[B^n], [P]$ and $[Q]$: since geodesics in $\oshape^{n*}$ are convex combinations, we can write
$$
[K_1] = [\alpha_1 B^n + \beta_1 P + \gamma_1 Q] \mbox{ and } [K_2] = [\alpha_2 B^n + \beta_2 P + \gamma_2 Q]~,
$$
where the $\alpha_i,\beta_i,\gamma_i$ are non-negative real numbers, with $\alpha_1 + \beta_1 + \gamma_1 = \alpha_2 + \beta_2 + \gamma_2 =1$. We want to prove that $\bd(\llbracket K_1 \rrbracket, \llbracket K_2 \rrbracket)=\d([K_1],[K_2])$, which means that for any $\Phi \in O(n)$ we have $\d([K_1],[K_2]) \leq \d([K_1],\Phi[K_2])$. Since $V_2$ is $O(n)-$invariant, we only need to show that
\begin{equation} \label{eq isometries triangles}
V_2(K_1,K_2) \leq V_2(K_1,\Phi(K_2))
\end{equation}
($K_1$ and $K_2$ denote two convex bodies in the equivalent classes $[K_1]$ and $[K_2]$). We have
$$
V_2(K_1,K_2) =
\alpha_1 \alpha_2 V_2(B^n) + \alpha_1 \beta_2 V_2(B^n,P) + \alpha_1 \gamma_2 V_2(B^n,Q)$$
$$
+\beta_1 \alpha_2 V_2(P,B^n) + \beta_1 \beta_2 V_2(P) + \beta_1 \gamma_2 V_2(P,Q)$$
$$
+\gamma_1 \alpha_2 V_2(Q,B^n) + \gamma_1 \beta_2 V_2(Q,P) + \gamma_1 \gamma_2 V_2(Q)~.
$$
Moreover $\Phi(K_2) = \alpha_2 B^n + \beta_2 \Phi(P) + \gamma_2 \Phi(Q)$, hence
$$
V_2(K_1,\Phi(K_2)) =
\alpha_1 \alpha_2 V_2(B^n) + \alpha_1 \beta_2 V_2(B^n,\Phi(P)) + \alpha_1 \gamma_2 V_2(B^n,\Phi(Q))$$
$$
+\beta_1 \alpha_2 V_2(P,B^n) + \beta_1 \beta_2 V_2(P,\Phi(P)) + \beta_1 \gamma_2 V_2(P,\Phi(Q))$$
$$
+\gamma_1 \alpha_2 V_2(Q,B^n) + \gamma_1 \beta_2 V_2(Q,\Phi(P)) + \gamma_1 \gamma_2 V_2(Q,\Phi(Q))~.
$$
And we obviously have $V_2(B^n,P)=V_2(B^n,\Phi(P))$ and  $V_2(B^n,Q)=V_2(B^n,\Phi(Q))$. Moreover,  Alexandrov--Fenchel Inequality \eqref{eq alex fench} gives $V_2(P)=\sqrt{V_2(P) V_2(\Phi(P))} \leq V_2(P,\Phi(P))$, and $V_2(Q)=\sqrt{V_2(Q) V_2(\Phi(Q))} \leq V_2(Q,\Phi(Q))$. And $\bd (\llbracket P \rrbracket, \llbracket Q \rrbracket)=\d([P],[Q])$ gives $V_2(P,Q) \leq V_2(P,\Phi(Q))$ and $V_2(Q,P) \leq V_2(Q,\Phi(P))$. Since all the real numbers $\alpha_i,\beta_i,\gamma_i$ are non-negative, this gives inequality \eqref{eq isometries triangles}.

\end{proof}

\subsection{Proof of Theorem~\ref{main2}}

Proposition \ref{prop:shapeCBBproper} and sections \ref{sectionnonuniqueness} and \ref{sectionembeddinghyp} give part of Theorem~\ref{main2}. It remains to prove the assertion about the boundary of $\shape^{n*}$. It obviously contains only one point: indeed, the boundary of $\oshape^{n*}$ is the set of segments up to homotheties, so the boundary of $\shape^{n*}$ is the set of segments, up to translations, positive scaling and rotations of $\R^n$, and there is only one equivalence class.

\section{The space of all the (oriented) shapes}\label{infini}

This section is an opening to the study of spaces of convex bodies, considered without making distinction between dimensions.
For $p\geq 0$, let us denote by $\iota_{n,p}$ the  canonical isometric embedding of $\R^n$ into $\R^{n+p}$ which is given by $\R^n \simeq \R^n \times \{0\}^{p} \subset \R^{n+p}$. 
Due to the intrinsic nature of $V_2$, we have that 
the map
$$\iota_{n,p}:(\oshape^{n*},\d) \to (\oshape^{(n+p)*},\dpk)$$
defined by $\iota_{n,p}([K])=[\iota_{n,p}(K)]$ is an isometry.
Let $\oshape^{\infty*}$ be the union over $n$ of $\oshape^{n*}$, quotiented by the following equivalence relation:  $[K_1]$ is equivalent to $[K_2]$ if and only if there exist $i,j \leq p$ such that  $K_1\subset\R^i$, $K_2\subset \R^j$ and $[\iota_{i,p-i}(K_1)]=[\iota_{j,p-j}(K_2)]$.
We will denote by $[K]_\infty$ an element of $\oshape^{\infty*}$. 
For two representatives of  $[K_1]_\infty, [K_2]_\infty \in \oshape^{\infty*}$ in $\R^n$, let us define
$$\di([K_1]_\infty,[K_2]_\infty)=\d([K_1],[K_2])~.$$
It is easy to see that $\di$ is well-defined and that it is actually a distance on $ \oshape^{\infty*}$.
The isometric embeddings $\iota_{n,p}$ induce isometric maps 
from $(\shape^{n*},\bd)$ to $(\shape^{(n+p)*},\bdpk)$, so in the same way we can define the set $\shape^{\infty*}$ and the metric space
$(\shape^{\infty*},\bdi)$. 

It follows from Theorems \ref{main} and \ref{main2} that $ (\oshape^{\infty*}, \di)$ and $(\shape^{\infty*},\bdi)$ 
are geodesic metric spaces. But two facts occur:
\begin{enumerate}[nolistsep]
\item it may happen that a sequence of convex bodies with non-empty interior in $\R^p$ converges to a convex body in $\oshape^\infty$ when $p$ goes to infinity. Actually, for   
$(\epsilon_p)_p$  a sequence of real numbers such that $\sqrt{p} \epsilon_p \rightarrow 0$,  one can check that the sequence $([\iota_{n,p}(K) + \epsilon_p B^{n+p}]_\infty)_p$ converges in $\oshape^{\infty*}$ to $[K]_\infty$.  In particular, there may exist other shortest paths than the convex combinations;
\item one can check that  the sequence of balls $([B^n]_\infty)_n$ (resp. $(\llbracket B^n \rrbracket_\infty)_n$) is a diverging Cauchy sequence.
\end{enumerate}

So we address the following.

\begin{question}
Describe the completion of $(\oshape^{\infty*},\di)$ and $(\shape^{\infty*},\bdi)$.
\end{question}

\begin{footnotesize}
\bibliographystyle{abbrv}

\end{footnotesize}

\end{document}